\documentclass[12pt]{amsart}
\usepackage{amsfonts}

\usepackage{amsmath}

\usepackage{amssymb}
\usepackage{amsthm}

\renewcommand{\Re}{\operatorname{Re}}

\newcommand{\Perm}{\rm{Perm}}

\newcommand{\reals}{{\mathbb{R}}}
\newcommand{\R}{{\mathbb{R}}}

\newcommand{\eps}{{\varepsilon}}

\newcommand{\nequiv}{{\not{\!\!\equiv}}}

\newcommand{\scriptd}{{\mathcal{D}}}
\newcommand{\scriptg}{{\mathcal{G}}}
\newcommand{\scriptt}{{\mathcal{T}}}

\newcommand{\scriptb}{{\mathcal{B}}}
\newcommand{\scriptp}{{\mathcal{P}}}

\newcommand{\bfi}{{\mathbf{i}}}

\newtheorem{theorem}{Theorem}
\newtheorem{lemma}[theorem]{Lemma}

\newtheorem*{???}{???}
\newtheorem*{cor*}{Corollary}
\newtheorem*{claim*}{Claim}
\newtheorem*{defn*}{Definition}

\begin{document}


\title[Endpoint $L^p \to L^q$ bounds]{Endpoint $L^{p} \to L^{q}$ bounds for integration along certain polynomial curves}
\thanks{The author was supported in part by NSF grants DMS-040126 and DMS-0901569.}

\author{Betsy Stovall}
\address{UCLA Mathematics Department\\
Box 951555\\
Los Angeles, CA 90095-1555
}

\email{betsy@math.ucla.edu}

\begin{abstract}
We establish strong-type endpoint $L^p(\reals^d) \to L^q(\reals^d)$ bounds for the operator given by convolution with affine arclength measure on polynomial curves for $d \geq 4$.  The bounds established depend only on the dimension $d$ and the degree of the polynomial.  
\end{abstract}

\maketitle

%%%%%%%%%%%%%%%%%%%%%%%%%%%%%%%%%%%%%%%%%
\section{Introduction} \label{sec:intro}

Let $P:\reals \to \reals^d$ be a polynomial and define the operator $T_P$ by
\begin{align} \label{eqn:def_op}
T_Pf(x) := \int_{I} f(x-P(s)) \, d\sigma_P(s),
\end{align}
where $I$ is an interval and $d\sigma_P$ represents affine arclength measure along $P$,
\begin{align} \label{def:affine_arclength}
d\sigma_P(s) := |\det(P'(s),P''(s),\ldots,P^{(d)}(s))|^{2/d(d+1)}\, ds.
\end{align}
The goal of this article is to establish $L^p \to L^q$ bounds for $T_P$ in the full conjectured range of exponents in dimensions $d \geq 4$ (together with a slight improvement in Lorentz spaces).  The conjectured range of $(p,q)$ depends only on $d$, and our bounds on the operator norms depend only on $p$, $q$,  $d$, and the degree of $P$.  This result has already been established in dimension 2 by Oberlin \cite{Oberlin02} and in dimension 3 by Dendrinos, Laghi, and Wright \cite{DLW}.  

\begin{theorem} \label{thm:main}
Let $d \geq 4$, let $P:\reals \to \reals^d$ be a polynomial of degree $N$, and let $T_P$ be the operator defined in \eqref{eqn:def_op}.  Let $p_d := \frac{d+1}{2}$ and $q_d:= \frac{d+1}{2}\frac{d}{d-1}$.  Then $T_P$ maps $L^p \to L^q$ if $(p,q) = (p_d,q_d)$ or $(q_d',p_d')$, with bounds depending only on $d,N$.  Moreover, $T_P$ maps the Lorentz space $L^{p_d,u}(\reals^d)$ boundedly into $L^{q_d,v}(\reals^d)$ and $L^{q_d',v'}(\reals^d)$ into $L^{p_d',u'}(\reals^d)$ whenever $u<q_d$, $v>p_d$, and $u<v$.
\end{theorem}

In the case when $I$ has infinite length and $d\sigma_P \nequiv 0$, this theorem is sharp up to Lorentz space endpoints.  A proof of this in the case when $P(t) = (t,t^2,\ldots,t^d)$ is given in \cite{StoCCC}.  

If $I$ has finite length, then $T_P$ is easily seen to be bounded from $L^1$ to $L^1$ and $L^{\infty}$ to $L^{\infty}$ (though finite $L^p \to L^q$ bounds depending only on $d,N$ cannot hold when $(p^{-1},q^{-1})$ lies off the line segment joining $(p_d^{-1},q_d^{-1})$ and $(1-q_d^{-1},1-p_d^{-1})$).  By interpolation with the bounds established in Theorem~\ref{thm:main}, we obtain nearly sharp bounds in this case as well.

We now give a little of the history of this problem.  It was observed by Drury in \cite{Drury90} that the affine arclength $d\sigma_P$ is in some ways a more natural choice than euclidean arclength for averages along $P$ and for the restriction of the Fourier transform to $P$.  For one, affine (like euclidean) arclength measure is easily seen to be parametrization independent; that is, if $\psi:\R \to \R$ is a diffeomorphism, then
$$
d\sigma_{P \circ \psi}(s) = |\psi'(s)||\det(P'(\psi(s)),\ldots,P^{(d)}(\psi(s)))|^{2/d(d+1)}\, ds.
$$
For two, affine (unlike euclidean) arclength measure compensates for degeneracies in the curve $P$ by vanishing where the torsion vanishes (for instance at the cusp in the curve parametrized by $(t^2,t^3)$).  This results in an $L^p \to L^q$ mapping theory in which $p$ and $q$ are  independent of the curve $P$.  Finally, as its name implies, affine (again unlike euclidean) arclength measure behaves nicely under affine transformations of $\R^d$.  This property will allow us to prove bounds on the operator norms of $T$ that depend only on $d$ and the degree of $P$.  A general discussion of affine arclength may be found in \cite{Gugg}.

Drury \cite{Drury90} established $L^p \to L^q$ bounds for the optimal range of $p,q$ in dimension 2 for $P(t) = (t,p(t))$ satisfying certain regularity conditions (though not necessarily polynomial).  Later, Oberlin \cite{Oberlin02} strengthened Drury's result in the polynomial case by establishing optimal $L^p \to L^q$ bounds for arbitrary polynomials in dimension 2 and for polynomials of the form $P(t) = (t,p_1(t),p_2(t))$ in dimension 3.  
In dimension 2, Oberlin established bounds for the operator norms of $T_P$ depending only on $p$, $q$, and the degree of $P$; he remarked that such invariant bounds were likely to hold in higher dimensions as well.  Recently, Dendrinos, Laghi, and Wright \cite{DLW} used a geometric inequality established in \cite{DW} to treat the general three dimensional case (also establishing invariant bounds).  Other results in a similar vein are due to Choi in \cite{Choi_JAustMS}, \cite{Choi_JKMS} and by Pan in \cite{Pan}.    A recent preprint of Oberlin (\cite{Oberlin_arxiv}) establishes bounds for $T_P$ in certain non-polynomial cases for $d=2,3,4$.  

More broadly, there is a growing body of literature on so-called generalized Radon transforms, operators defined by integration over families of submanifolds of $\reals^d$.  We mention here a few articles which are particularly closely related to this one; by no means is this an exhaustive bibliography.  In \cite{ChCCC}, Christ developed a combinatorial technique that he used to establish optimal restricted weak type estimates for convolution with affine arclength measure on the {\em{moment curve}} $(t,t^2,\ldots,t^d)$ for $d \geq 4$ (strong type bounds had been proved in lower dimensions by Littman \cite{Litt} and Oberlin \cite{Obe1,Obe2,Obe3} via different techniques).  Later on, Tao and Wright used ideas from \cite{ChCCC} together with several new ideas to establish optimal (up to Lebesgue space endpoint) bounds for averages over families of smooth curves in \cite{TW}.  This article is one of a few recent efforts toward establishing the endpoint bounds in special cases of the Tao--Wright theorem.  In \cite{GressPolyEndpt}, Gressman proved the restricted weak-type estimate at the endpoint in the polynomial case of Tao and Wright's theorem.   In \cite{ChQex}, Christ showed that it was possible to use arguments similar to those in \cite{ChCCC} to prove strong-type estimates.  We mention three subsequent applications of this technique (this article being a fourth), namely \cite{DLW} wherein Dendrinos, Laghi, and Wright consider $T_P$ in dimension 3, \cite{Laghi} in which Laghi proves endpoint bounds for a restricted $X$-ray transform, and the author's \cite{StoCCC} which establishes strong-type bounds for convolution with affine arclength measure along the moment curve in dimensions $d\geq 4$ (as mentioned above, low-dimensional results were already known).  In this article, we will use techniques from \cite{ChCCC}, \cite{ChQex}, \cite{DLW}, and \cite{StoCCC}, mentioned above, as well as from \cite{Gress06}, in which Gressman established restricted weak-type bounds for convolution with certain measures along polynomial curves whose entries are monomials.

Related to our problem is the restriction of the Fourier transform to curves with affine arclength.  In this case, it is conjectured that uniform $L^p \to L^q$ bounds hold, and though considerable progress has been made, the conjecture has not been completely resolved.  A detailed history of this problem may be found in \cite{DW}, for instance, along with some recent results.  Other articles in this vein include the recent work of Bak, Oberlin, and Seeger \cite{BOSI}, \cite{BOSII} and the less recent work of Drury and Marshall \cite{DMI}, \cite{DMII} and of Sj\"olin \cite{Sjolin} (and of course Drury's \cite{Drury90}).

%%%%%%%%%%%%%%%%%%%
\subsection*{Notation}  In this article, we will write $A \lesssim B$ when $A \leq C B$ with the constant $C$ depending only on the ambient dimension and the degree of the polynomial $P$.  We will also use the accompanying notation $A \sim B$.  We will occasionally use the `big $O$' notation, writing $A = O(B)$ instead of $A \lesssim B$.  In addition, if $1 \leq p \leq \infty$, $p'$ refers to the exponent dual to $p$ (which satisfies $1=p^{-1}+{p'}^{-1}$).

%%%%%%%%%%%%%%%%%%%
\subsection*{Acknowledgements}  This article was adapted from part of the author's Ph.D.\ thesis, and she would like to thank her advisor, Mike Christ, for initially drawing her attention to the article \cite{DLW} and for his continuing support and advice.  In particular, Prof.~Christ's kind proofreading of and comments on earlier versions of this manuscript significantly improved the exposition.  The author would also like to thank the anonymous referee at the JFA for many valuable comments and suggestions.

%%%%%%%%%%%%%%%%%%%%%%%%%%%%%%%%%%%%%%%%%%%%%
\section{Initial simplifications and a key theorem} \label{sec:reductions}

As mentioned above, our problem is closely related to the problem of the restriction of the Fourier transform to curves with affine arclength.  One of the main tools used here and in \cite{DLW} is a theorem which was originally proved by Dendrinos and Wright in \cite{DW} as one step toward such a restriction theorem.  

Let 
\begin{align} \label{def:L_P}
L_P(s) := \det(P'(s),P''(s),\ldots,P^{(d)}(s)),
\end{align}
for $s \in \reals$.  Hence $d\sigma_P = L_P^{2/d(d+1)}\,ds$.  For $t = (t_1,t_2,\ldots,t_d) \in \reals^d$, define
\begin{align} \label{def:J_P}
J_P(t) := \det(P'(t_1),P'(t_2),\ldots,P'(t_d)).
\end{align}
As indicated by the notation, this last term will arise as the jacobian of a certain map from $\reals^d \to \reals^d$.

\begin{theorem}[Dendrinos and Wright]\label{thm:DW}  Let $P:\reals \to \reals^d$ be a polynomial of degree $N$ such that $L_P(s) \not{\!\!\equiv} 0$.  Then there exists a decomposition $\reals = \bigcup_{j=1}^{C_N} \overline{I_j}$ such that the $I_j$ are pairwise disjoint open intervals, and for each $j$, there exist a positive constant $A_j$, a non-negative integer $K_j \leq N$, and a real number $b_j \in \reals \backslash I_j$ such that
\begin{enumerate}
\item $|L_P(s)| \sim A_j|s-b_j|^{K_j}$ for every $s\in I_j$ 
\item Whenever $(t_1,\ldots,t_d) \in I_j^d$,
\[
|J_{P}(t_1,\ldots,t_d)| \geq C \prod_{k=1}^d |L_P(t_k)|^{1/d} \prod_{l<k}|t_k-t_l|
\]
\item For any $\eps \in \{-1,1\}^d$, define $\Phi^{\eps}(t_1,\ldots,t_d) := \sum_{j=1}^d \eps_j P(t_j)$.  Then for each $x \in \reals^d$, the cardinality of $(\Phi^{\eps})^{-1}\{x\} \cap I_j^d$ is at most $d!$.
\end{enumerate}
Here $C$ and the implicit constants depend only on $N$ and $d$.
\end{theorem}

We may assume in proving Theorem \ref{thm:main} that $I$ has finite length.  Theorem \ref{thm:DW} allows us to make a few further simplifications.  First, it suffices to prove the desired bounds for an operator as in \eqref{eqn:def_op}, only with the integral restricted to one of the intervals $I_j$ from the decomposition above.  Next, after translating $I_j$ and reflecting it across 0 if needed, we may assume that the real number $b_j$ in Theorem \ref{thm:DW} is equal to 0 and that $I_j \subset (0,\infty)$.  The $L^{p_d,u} \to L^{q_d,v}$ bounds are invariant under scalings in $s$ ($P^a(s) = P(as)$) and multiplication of $P$ by a constant.  By scaling in the $s$ variable, we may assume that $|I_j| = 1$.  Finally, using the fact that $L_{\lambda P} = \lambda^d L_P$, after multiplying $P$ by an appropriate constant, we may assume that the constant $A_j$ in Theorem \ref{thm:DW} is equal to 1.  In particular, $|L_P(s)| \sim s^K$ uniformly on $I$, where the implicit constants depend only on $d,N$.  In summary, it suffices to prove uniform estimates for the operator
\begin{align} \label{eq:def_T2}
Tf(x) := \int_I f(x-P(s)) \, d\mu(s),
\end{align}
where $d\mu(s) = s^{2K/d(d+1)}\, ds$ and $I \subset (0,\infty)$ has length 1.

We note that Dendrinos, Folch-Gabayet, and Wright have recently shown in \cite{DF-GW} that Theorem \ref{thm:DW} extends to $d$-tuples of rational functions of bounded degree.  It therefore seems likely that Theorem \ref{thm:main} could be generalized to give bounds for convolution with affine arclength along curves parametrized by such functions, but the author has not investigated the extent to which the arguments in this article would need to be changed.

%%%%%%%%%%%%%%%%%%%%%%%%%%%%%%%%%%%
\section{Reduction of Theorem \ref{thm:main} to two lemmas}

If $E$ and $F$ are subsets of $\reals^d$, then we will denote
\[
\langle T \chi_E, \chi_F \rangle =: \scriptt(E,F).
\]
We define two quantities,
\begin{align} \label{def:alpha_beta}
\alpha:= \frac{\scriptt(E,F)}{|F|} \qquad \beta := \frac{\scriptt(E,F)}{|E|},
\end{align}
which have played an important role in much of the recent literature on generalized Radon transforms and, in particular, appeared in most of the references mentioned in Section \ref{sec:intro}.  Note that $\beta$, for instance, represents the average over $x \in E$ of 
\[
T^*\chi_F(x) = \mu\{s \in I: x+P(s) \in F\}.
\]

The quantities $\alpha,\beta$ have mostly been used to prove restricted weak type bounds, i.e. $\scriptt(E,F) \lesssim |E|^{1/p}|F|^{1/q'}$.  One can easily check that in our case, the restricted weak-type bound at the endpoint would follow from 
\[
|E| \gtrsim \alpha^{d(d+1)/2}(\beta/\alpha)^{d-1}.
\]
In \cite{ChQex}, Christ proved that ``trilinear'' versions of such estimates could be used to establish strong-type bounds.  This method was applied in \cite{DLW} to prove strong-type estimates for the operator in \eqref{eqn:def_op} in the cases $d=2,3$ and in \cite{StoCCC} to prove strong-type estimates for convolution with affine arclength measure along the moment curve.  Using Christ's techniques, one can show that the next two lemmas imply Theorem \ref{thm:main}.  See \cite{StoCCC} for details.  For notational convenience, we define the quantity 
\begin{align}\label{def:n}
n := d(d+1)/(2K+d(d+1)).
\end{align}
Notice that for $0 \leq r \leq 1$,
\begin{align} \label{eqn:mu(0,r)}
\mu([0,r^n]) = n r.
\end{align}

\begin{lemma} \label{lemma:mlE}
Let $E_1,E_2,F \subset \reals^d$ have finite positive measures.  Assume that for $j=1,2$ $T\chi_{E_j}(y)  \geq \alpha_j$ for all $y \in F$ and $\frac{\mathcal{T}(E_j,F)}{|E_j|}  \geq \beta_j$, where $\alpha_2 \geq \alpha_1$.  Then
\begin{align}\label{eq:mlE}
    |E_2| &\gtrsim 
    \alpha_1^{\frac{d(d+1)}{2}}(\frac{\beta_1}{\alpha_1})^{d-1} (\alpha_2/\alpha_1)^{(d+1)/2+n(d-1)/2}.
\end{align}
\end{lemma}

\begin{lemma} \label{lemma:mlF}
Let $\eta > 0$, and let $E,F_1,F_2 \subset \reals^d$ have finite, nonzero measures.  Assume that for $j=1,2$ $T^*\chi_{F_j}(x) \geq \beta_j$ for all $x \in E$, where $\beta_j \gtrsim \eta \frac{\scriptt(E,F_j)}{|E|}$ and $\beta_2 \geq \beta_1$.  Assume as well that $\frac{\mathcal{T}(E,F_j)}{|F_j|} \geq \alpha_j$, where $\alpha_2\leq\alpha_1$.  Then 
\begin{align}\label{eq:mlF}
    |F_2| &\gtrsim \eta^C \alpha_1^{r_1}\alpha_2^{r_2}\beta_1^{s_1}\beta_2^{s_2},
\end{align}
for some quadruple $(r_1,r_2,s_1,s_2)$ which is taken from a finite list that depends on $d,N$ and which satisfies
\begin{align} \label{eq:ineqexp} 
 \frac{d(d-1)}{2}= r_1+r_2 , \qquad d=s_1+s_2, \qquad 0<\frac{s_2}{q_d'}-\frac{r_2}{q_d}-1.
\end{align}
The constant $C$ is allowed to depend on $N,d$ alone.
\end{lemma}

In some sense, Lemma \ref{lemma:mlF} corresponds to a more general formulation.  The relative simplicity of Lemma \ref{lemma:mlE} is a matter of luck more than anything else.  See \cite{StoCCC} for more explanation.  We remark that Lemma \ref{lemma:mlF} is slightly different from the corresponding lemma in \cite{StoCCC}, but implies the strong-type bound (and accompanying Lorentz space improvement) by exactly the same proof.

%%%%%%%%%%%%%%%%%%%%%%%%%%%%%%%%%%%%%%%%%%%
\section{The proof of Lemma \ref{lemma:mlE}}

%%%%%%%%%%%
\subsection{Setup} 

For $i \geq 1$, define 
\[
\Phi_i(t_1,\ldots,t_i) := \sum_{j=1}^i (-1)^{j+1} P(t_i).
\]
Our first step will be to identify sets $\Omega_i$, satisfying certain helpful properties, such that for each $i$, $\Phi_i(\Omega_i)$ is contained in $F$ or one of the $E_j$.  

We assume the hypotheses of Lemma \ref{lemma:mlE} and define 
\[
\gamma_1 := \max\{\alpha_1,\beta_1\}.
\]

\begin{lemma} \label{lemma:setup_mlE}
There exist a point $x_0 \in E_1$, a constant $c > 0$, and measurable sets $\Omega_1 \subset I$, $\Omega_i \subset \Omega_{i-1} \times I$ for $2 \leq i \leq 2d$ such that
\begin{itemize}
\item $x_0 + \Phi_i(\Omega_i)$ is contained in:  $F$ if $i$ is odd, $E_2$ if $i=2d$, and $E_1$ otherwise.
\item $\mu(\Omega_1) = c\beta_1$, and if $t \in \Omega_{i-1}$, then $\mu\{t_i \in I : (t,t_i) \in \Omega_i\}$ equals:  $c\beta_1$ if $i>1$ is odd, $c\alpha_2$ if $i=2d$, and $c\alpha_1$ otherwise.
\item If $t \in \Omega_i$, then $t_i$ is greater than or equal to:  $c \gamma_1^n$ if $1 \leq i < 2d$ and $c \alpha_2^n$ if $i=2d$.  Furthermore, if $j < i<2d$, then $|t_i-t_j|$ is greater than or equal to: $c \beta_1 t_j^{-2K/d(d+1)}$ if $i$ is odd and $c \alpha_1 t_j^{-2K/d(d+1)}$ if $i<2d$ is even.  Finally, if $j < i=2d$, then $|t_i-t_j|$ is greater than or equal to $c' t_i$ if $t_j \lesssim \alpha_2^n$ and $c' \alpha_2 t_j^{-2K/d(d+1)}$ otherwise.
\end{itemize}
\end{lemma}

\begin{proof}  This is quite similar to Lemma 1 of \cite{ChCCC}, but the last item involves changes inspired by the works \cite{Gress06} and \cite{DLW}.  

First, suppose that we have proved the existence of sets $\Omega_i$ as described in the lemma for $1 \leq i \leq 2d-1$.  Then every point $t \in \Omega_{2d-1}$ corresponds to a point $y(t) = x_0 + \Phi_{2d-1}(t) \in F$.  By the assumed lower bound $T\chi_{E_2}(y) \geq \alpha_2$ on $F$, there exists a set $I(t) \subset I$ with $\mu(I(t)) \geq \alpha_2$ such that $y(t)-P(I(t)) \in E_2$. 

Next, we refine the sets $I(t)$ to guarantee the third condition of the lemma.  Since $\mu[0,c\alpha_2^n] \sim c^{1/n} \alpha_2$, we can assume that $s \geq c\alpha_2^n$ for $s \in I(t)$ while maintaining $\mu(I(t)) \gtrsim \alpha_2$.  With this assumption in place, let $j < 2d$.  If $t_j < c \alpha_2^n$, then we excise the region $[0,2c\alpha_2^n]$ from $I(t)$ if necessary.  Having done that, if $s \in I(t)$, then $s-t_j \geq s/2$.  We claim that if $t_j > c \alpha_2^n$, then $\mu(B(t_j)) := \mu\{s \in I(t) : |s-t_j| < c' \alpha_2 t_j^{-2K/d(d+1)}\} < c \alpha_2$, for $c'$ sufficiently small.  To see this, note that if $s \in B(t_j)$, then $|s-t_j| < c' \alpha_2^n < c' s$, so $s \sim t_j$ and 
\[
\mu(B(t_j)) = \int_{t_j-c'\alpha_2 t_j^{-2K/d(d+1)}}^{t_j+c'\alpha_2 t_j^{-2K/d(d+1)}} s^{2K/d(d+1)}\, ds \leq c \alpha_2 t_j^{2K/d(d+1)-2K/d(d+1)}.
\]
Thus, by removing small ``bad'' portions of each $I(t)$ if needed (this may be done while preserving measurability), we may set $\Omega_{2d} = \{(t,s) \in \Omega_{2d-1} \times I : s \in I(t)\}$, and satisfy the conclusions of the lemma.

With some modifications similar to those indicated above, the proof of the existence of sets $\Omega_i$, $1 \leq i \leq 2d-1$ is essentially the same as the proof of Lemma 1 in \cite{Gress06}.  The key observation is that $\langle T_{\gamma_1} \chi_{E_1}, \chi_{F} \rangle \gtrsim \scriptt(E_1,F)$, where 
\begin{align} \label{def:T_gamma}
T_{\gamma_1} f(x) = \int_{I \backslash [0,c\gamma_1^n]} f(x-P(s)) \, d\mu(s).
\end{align}
The proof then proceeds as in Lemma 1 of \cite{ChCCC}.
\end{proof}

%%%%%%%%%%%%%%%%%%%%%%%%%%%%%
\subsection{The proof when $\beta_1 \gtrsim \alpha_1$.}

Let $t_0 \in \Omega_d$, and let $\omega := \{t \in I^d : (t_0,t) \in \Omega_{2d}\}$.  For consistency of notation, we will refer to elements of $\omega$ as $t = (t_{d+1},\ldots,t_{2d})$.  

In this section, we will prove the following lemma.

\begin{lemma}\label{lemma:initial_lb}
Let $r_d = 1+ 3 + \cdots + d-1$ if $d$ is even and $r_d = 2 + 4 + \cdots + d-1$ if $d$ is odd.  Then
\[
|E_2| \gtrsim \alpha_1^{d(d+1)/2}(\beta_1/\alpha_1)^{r_d} (\alpha_2/\alpha_1)^{(d+1)/2+n(d-1)/2}
\]
\end{lemma}

Note that $r_d \geq d-1$, so when $\beta_1 \gtrsim \alpha_1$, the conclusion of Lemma \ref{lemma:initial_lb} implies Lemma \ref{lemma:mlE}.

\begin{proof}[Proof of Lemma \ref{lemma:initial_lb}]  Let $\Phi(t) := x_0 + \Phi_{2d}(t_0,t)$.  Note that (iii) of Theorem \ref{thm:DW} implies that 
\[
|E_2| \gtrsim \int_{\omega} | \det D\Phi(t)|\, dt = \int_{\omega} |J_P(t)|\, dt.
\]
Hence, by (i--ii) of Theorem \ref{thm:DW} and the fact that $\omega \subset I^d$,
\begin{align} \label{ineq:initial_lb}
|E_2|  \gtrsim \int_{\omega} \prod_{k=d+1}^{2d} \prod_{d<l<k} t_k^{K/d(d-1)}t_l^{K/d(d-1)}|t_k-t_l|\, dt.
\end{align}
We know that $\mu^d(\omega) \gtrsim \alpha_1^{\lceil d/2 \rceil}\beta_1^{\lfloor d/2 \rfloor}(\alpha_2/\alpha_1)$, and our next task will be prove a lower bound for \eqref{ineq:initial_lb} which involves integrals with respect to $\mu$.  It will suffice to prove the following:

\begin{lemma}\label{lemma:elim_diff} 
For every $t \in \omega$ and $k \leq 2d$,
\begin{itemize}
\item If $k$ is odd and $j<k$, then:  $|t_k-t_j| \gtrsim \beta_1 (t_k\cdot t_j)^{-K/d(d+1)}$.
\item If $k<2d$ is even and $j < k$, then:  $|t_k-t_j| \gtrsim \alpha_1 (t_k\cdot t_j)^{-K/d(d+1)}$.
\item Finally, if $j<k=2d$, then:  $|t_k-t_j| \gtrsim \alpha_2^{(1+n)/2} \alpha_1^{(1-n)/2}(t_k\cdot t_j)^{-K/d(d+1)}$.
\end{itemize}
\end{lemma}

Separating the integrand in \eqref{ineq:initial_lb} into a product over even $k$ times a product over odd $k$, and manipulating the inequalities in Lemma \ref{lemma:elim_diff}, we see that
\begin{align*}
|E_2| &\gtrsim \alpha_1^{r_d} \beta_1^{r_{d-1}} (\frac{\alpha_2}{\alpha_1})^{(n+1)(d-1)/2} \int_{\omega} \prod_{k=d+1}^{2d} t_k^{2K/d(d+1)}\, dt  \\ &\gtrsim \alpha_1^{d(d+1)/2}(\frac{\beta_1}{\alpha_1})^{r_d} (\frac{\alpha_2}{\alpha_1})^{(d+1)/2 + n(d-1)/2}.
\end{align*}
The last inequality also requires a little algebra.
\end{proof}

\begin{proof}[Proof of Lemma \ref{lemma:elim_diff}.]  
We will prove the lemma when $k=2d$.  The proof when $k<2d$ is similar, but a little simpler.  

Suppose that $t_j \ll t_k$ or $t_j \gg t_k$.  Then by the triangle inequality, $|t_k-t_j| \gtrsim t_k$, and the lower bounds in Lemma \ref{lemma:setup_mlE} imply that $t_k^{1+K/d(d+1)}t_j^{K/d(d+1)} \gtrsim \alpha_2^{(1+n)/2}\alpha_1^{(1-n)/2}$.  

If $t_j \sim t_k$, then $|t_j-t_k| \gtrsim \alpha_2 t_j^{-2K/d(d+1)} \sim \alpha_2 (t_kt_j)^{-K/d(d+1)}$ and the claimed inequality follows from $\alpha_2 \geq \alpha_1$.
\end{proof}

This completes the proof of Lemma \ref{lemma:mlE} in the case $\beta_1 \gtrsim \alpha_1$.

%%%%%%%%%%%%%%%%%%%%%%%%%%%%%%%%%%%%
\subsection{A combinatorial argument.}\label{subsec:combin}

To prove Lemma \ref{lemma:mlE} in the remaining case, $\beta_1 \ll \alpha_1$, we modify the ``band structure'' argument of Christ in \cite{ChCCC} to accommodate the weight $s^{2K/d(d+1)}$ in the measure $\mu$.  The modifications we make are inspired by the arguments of \cite{DLW}.  We will follow Christ's construction as closely as possible.

\begin{defn*} Let $\Gamma$ be a finite set of positive integers, called indices.  A band structure on $\Gamma$ is a partition of $\Gamma$ into subsets called ``bands''.  Given a band structure on $\Gamma$, we designate the indices in $\Gamma$ as free, quasi-free, or bound as follows:
\begin{itemize}
\item The least index of each band is free.
\item An index is quasi-free if it is the larger element of a two-element band; it is quasi-bound (not bound) to the smaller (free) element of that band.  
\item An index is bound (to the free element of its band) if it is one of two or more non-free elements of some band.
\end{itemize}
\end{defn*}

\begin{lemma}\label{lemma:exists_b_struct} Let $\eps>0$.  Then there exist parameters $\delta, \delta'$ satisfying $c_{d,\eps} < \delta' < \eps \delta < \eps c$, an integer $d \leq k < 2d$, an element $t_0 \in \Omega_{2d-k}$, a set $\omega \subset \{t:(t_0,t) \in \Omega_{2d}\}$ with $\mu^{k}(\omega) \sim \alpha_1^{\lceil k/2 \rceil} \beta_1^{\lfloor k/2 \rfloor} (\alpha_2/\alpha_1)$, and a band structure on $[2d-k+1,2d]$ such that the following properties hold:
\begin{enumerate}
\item There are exactly $d$ free or quasi-free indices.  In particular, each even index is free.
\item $|t_i-t_j| > \delta \alpha_1 (t_i \cdot t_j)^{-K/d(d+1)}$, unless $i$ and $j$ lie in the same band.
\item $c_0 \beta_1 (t_i \cdot t_j)^{-K/d(d+1)} < |t_i-t_j| < \delta \alpha_1 (t_i \cdot t_j)^{-K/d(d+1)}$ whenever $i$ is quasi-bound to $j$.
\item $\delta' \alpha_1 (t_i \cdot t_j)^{-K/d(d+1)} > |t_i-t_j|$ whenever $i$ is bound to $j$.
\end{enumerate}
\end{lemma}

In our application of the lemma, the parameter $\eps$ will be the value whose existence is guaranteed by the forthcoming Lemma \ref{lemma:lb_J}.  

\begin{proof}[Sketch of proof of Lemma \ref{lemma:exists_b_struct}.]  Because of the lower bounds from Lemma \ref{lemma:elim_diff} and the fact that $\alpha_2 \geq \alpha_1$, the proof of this lemma follows from arguments similar to those in \cite{StoCCC}, which are in turn slight modifications of arguments in \cite{ChCCC}.  We sketch the beginnings of the modified argument.  

There exists $\omega_0 \subset \Omega_{2d}$ with $|\omega_0| \gtrsim |\Omega_{2d}|$ and a permutation $\sigma \in \rm{Perm}_{2d}$ (where $\Perm_j$ is the set of permutations on $j$ indices) such that 
\[
t_{\sigma(1)} < t_{\sigma(2)} < \cdots < t_{\sigma(2d)}
\]
on $\omega_0$.  Initially set $\delta =  \frac{c_{N,d}}{2d}$, where $c_{N,d}$ is some small constant to be determined below.  Refining $\omega_0$ if necessary, but retaining the above lower bound on $|\omega_0|$, there exists a sequence $1=L_1 < L_2 < \cdots < L_R \leq 2d$ such that $t_{\sigma(i-1)} < t_{\sigma(i)} - \delta \alpha_1 (t_{\sigma(i-1)}t_{\sigma(i)})^{-K/d(d+1)}$ if and only if $i=L_j$ for some $1 < j \leq R$.  

After this first step of the decomposition, our bands are 
\[
\{\sigma(L_1)=\sigma(1),\ldots, \sigma(L_2-1)\}, \{\sigma(L_2),\ldots,\sigma(L_3-1)\},\ldots,\{\sigma(L_R),\ldots,\sigma(2d)\}.
\]

Assume that $i<j$ and $\sigma(i),\sigma(j)$ are in the same band.  Then by induction,
\begin{align*}
t_{\sigma(j)} > t_{\sigma(i)} &\geq t_{\sigma(j)} - \delta \alpha_1 [(t_{\sigma(i)}t_{\sigma(i+1)})^{-K/d(d+1)} + \cdots  + (t_{\sigma(j-1)}t_{\sigma(j)})^{-K/d(d+1)}]\\
& > t_{\sigma(j)} - 2d \delta\alpha_1 t_{\sigma(i)}^{-2K/d(d+1)} \\
& \geq t_{\sigma(j)} - c_{N,d} \alpha_1t_{\sigma(i)}^{-2K/d(d+1)} \geq t_{\sigma(j)} - {c_{N,d}}' \alpha_1^n .
\end{align*}
The inequality on the second line follows from the monotonicity of the $t_{\sigma(\cdot)}$ and the second inequality on the last line follows from $t_{\sigma(i)} \gtrsim \alpha_1^n$ (recall $\alpha_2 > \alpha_1$) and a bit of algebra.  For $c_{N,d}$ sufficiently small (depending only on $N,d$), the second inequality on that line implies that $|t_{\sigma(i)}-t_{\sigma(j)}| \ll \min\{t_{\sigma(i)},t_{\sigma(j)}\}$, which in turn implies that $t_{\sigma(i)} \sim t_{\sigma(j)}$ (we will use this fact several times in the coming pages).    Returning to the first line of the sequence of inequalities, 
\[
t_{\sigma(j)} > t_{\sigma(i)} > t_{\sigma(j)} -  c_{N,d} \alpha_1 (t_{\sigma(i)}t_{\sigma(j)})^{-K/d(d+1)}.
\]
From this and the lower bounds in Lemma \ref{lemma:elim_diff}, the maximum of $\sigma(i)$ and $\sigma(j)$ cannot be even.  Thus an even index is always the minimum element of a band containing it, and in particular, no band contains two or more even indices after the first step.

Assume that $i<j$ and $\sigma(i),\sigma(j)$ are in different bands.  Say $\sigma(i)$ lies in the band
\[
\{\sigma(L_{a-1}),\ldots, \sigma(L_a -1)\}.
\]
Then $j \geq L_a$ and 
\begin{align*}
t_{\sigma(j)}-t_{\sigma(i)} &= \sum_{k=i}^{j-1} t_{\sigma(k+1)}-t_{\sigma(k)} \geq t_{\sigma(L_a)} - t_{\sigma(L_a - 1)} \\
& \geq \delta \alpha_1 (t_{\sigma(L_a)}t_{\sigma(L_a - 1)})^{-K/d(d+1)} \gtrsim \delta \alpha_1 (t_{\sigma(j)}t_{\sigma(i)})^{-K/d(d+1)}.
\end{align*}
The first inequality on the second line follows from the definition of the sequence $L_a$.  The second inequality follows from the monotonicity of the $t_{\sigma(\cdot)}$ and the fact that $t_{\sigma(i)} \sim t_{\sigma(L_a - 1)}$ because both lie in the same band (this was proved in the previous paragraph).

Note that it is vital that $\alpha_2 \geq \alpha_1$, so that $t_{2d} \gtrsim \alpha_1^n$ on $\Omega_{2d}$.

We have explained the first step of an iterative procedure which terminates after $O(d)$ steps.  In each step of this procedure, the quantities $\delta,\delta'$ are decreased, but this is only done finitely many times and in a way which guarantees the lower bound $c_{d,\eps} < \delta'$.  For the remainder of the algorithm, we refer the reader to \cite{StoCCC} (the arguments there are adapted from \cite{ChCCC}) with the caveat that adjustments as above will be necessary.  
\end{proof}

To state the next lemma, we need a little further notation.

Suppose $\omega$, $k$, and $t_0$ are as in the conclusion of Lemma \ref{lemma:exists_b_struct}.  With $t \in \reals^k$ denoted by $(t_{2d-k+1},\ldots,t_{2d})$, define $\Phi(t) := x_0 + \Phi_{2d}(t_0,t)$.  Thus $\Phi(\omega) \subset E_2$.  

Let $\Lambda \subset [2d-k+1,2d]$ be the set of free or quasi-free indices.  Given $i \in [2d-k+1,2d] \backslash \Lambda$, let $B(i)$ denote the index to which $i$ is bound.  We define three mappings:
\begin{align*}
\reals^{d} \ni \tau(t), \,\, \tau_i(t) &= t_i, \, \text{ for $i \in \Lambda$}\\
\reals^{k-d} \ni s(t), \,\, s_j(t) &= t_j-t_{B(j)}, \, \text{ for $j \notin \Lambda$}\\
\reals^{k-d} \ni \sigma(t), \,\, \sigma_j(t) &= s_j(t) [\tau_{B(j)}(t)]^{2K/d(d+1)}.
\end{align*}
The function $t \mapsto (\tau,\sigma)$ is invertible for $t \in (0,\infty)^k$.  We let $t(\tau,\sigma)$ denote its inverse, and define
\[
J(\tau,\sigma) := \det ( D_{\tau} \Phi(t(\tau,\sigma)).
\]

\begin{lemma}\label{lemma:lb_J}  There exists $\eps > 0$ such that given any $k$, any band structure on $[2d-k+1,2d]$, and any $\omega$ satisfying the conclusion of Lemma \ref{lemma:exists_b_struct}, the following lower bound holds for all $(\tau,\sigma) \in \reals^k$ such that $t(\tau,\sigma) \in \omega$:
\begin{align}\label{ineq:lb_J}
|J(\tau,\sigma)| \geq c \alpha_1^{d(d-1)/2}(\frac{\beta_1}{\alpha_1})^M(\frac{\alpha_2}{\alpha_1})^{(1+n)(d-1)/2} \prod_{j=1}^d\tau_j^{2K/d(d+1)},
\end{align}
for some $c>0$, where $M$ is the number of quasi-free indices in the band structure.  Here the values of $\eps, c$ depend only on $d$, $N$, and the constant $c_0$ in (iii) of Lemma \ref{lemma:exists_b_struct} (which in turn depends on $d,N$).
\end{lemma}

\begin{proof}[Completion of proof of Lemma \ref{lemma:mlE}]

The remainder of the proof is quite similar to the argument from \cite{ChCCC}.  Let $k$, $\omega$, etc., be as in Lemma \ref{lemma:exists_b_struct}, and fix $\sigma \in \reals^{k-d}$.  Let $\omega_{\sigma} = \{\tau : t(\tau,\sigma) \in \omega\}$.  By Bezout's theorem (as stated in \cite{Shaf}, for a similar application see \cite{ChCCC}), for each $\sigma$, under the map $\omega \ni \tau \mapsto \Phi(t(\tau,\sigma))\in E_2$, points $x$ lying off a set of measure zero have at most $C_{N,d}$ preimages.  Therefore
\begin{align}\label{ineq:local8}
|E_2| \gtrsim \int_{\omega_{\sigma}} |J(\tau,\sigma)| \, d\tau 
\gtrsim  \alpha_1^{d(d-1)/2}(\beta_1/\alpha_1)^M (\alpha_2/\alpha_1)^{(1+n)(d-1)/2} \mu^d(\omega_{\sigma}).
\end{align}
It is easy to check that if $t(\tau,\sigma) \in \omega$, then $|\sigma| \lesssim \alpha_1$.  Integrating both sides of \eqref{ineq:local8} over $|\sigma| \lesssim \alpha_1$,
\begin{align} \label{local87}
\alpha_1^{k-d}|E_2| \gtrsim \alpha_1^{d(d-1)/2}(\beta_1/\alpha_1)^M (\alpha_2/\alpha_1)^{(1+n)(d-1)/2} \int_{t(\tau,\sigma) \in \omega} d\mu^d(\tau)\,d\sigma.
\end{align}
We switch the order of integration and make the change of variables $(\tau,\sigma) \mapsto t$ to show that the integral on the right of \eqref{local87} equals
\begin{align} \label{ineq:local86}
\int_{\omega} \prod_{j \notin \Lambda} t_{B(j)}^{2K/d(d+1)} dt_j \prod_{i \in \Lambda} t_i^{2K/d(d+1)} dt_i.
\end{align}
We showed above (in the proof of Lemma \ref{lemma:exists_b_struct}) that $t_i \sim t_j$ whenever $i$ is bound to $j$ and $t \in \omega$.  Approximating the $t_{B(j)}$ in \eqref{ineq:local86} by $t_j$, we then have
\begin{align} \label{ineq:local85}
\alpha_1^{k-d} |E_2| \gtrsim \alpha_1^{d(d-1)/2}(\beta_1/\alpha_1)^M (\alpha_2/\alpha_1)^{(1+n)(d-1)/2} \int_{\omega} d\mu^k(t).
\end{align}
Next, applying (i) of Lemma \ref{lemma:exists_b_struct} and \eqref{ineq:local85}, we have shown that
\[
|E_2| \gtrsim \alpha_1^{d(d+1)/2} (\frac{\beta_1}{\alpha_1})^{M+\lfloor k/2 \rfloor}(\alpha_2/\alpha_1)^{(d+1)/2 + n(d-1)/2}.
\]
Finally, at least $\lfloor k/2 \rfloor + 1$ indices are free (the even indices together with the first index), and $M$ plus the number of free indices equals $d$.  Therefore $M + \lfloor k/2 \rfloor \leq d-1$, and we have proved Lemma \ref{lemma:mlE} in the remaining case, $\beta_1 \ll \alpha_1$.
\end{proof}

%%%%%%%%%%%%%%%%%%%%%%%%%%%%%%%%%%%%%%%%%%%%%%%%%%%
\section{The proof of Lemma \ref{lemma:lb_J}}
Now we return to the lower bound on $J(\tau,\sigma)$.

We recall some important information about the interval $I$ in \eqref{eq:def_T2}.  The interval came from the decomposition in Theorem \ref{thm:DW} and is contained in $(0,\infty)$ by the reductions in Section~\ref{sec:reductions}.  In addition, the corresponding quantity $b$ equals zero.  

Analysis similar to that in \cite{ChCCC} can be used to write $J(\tau,s)$ in a more helpful form.  We summarize the argument.  Note that
\begin{align*}
\Phi(t(\tau,\sigma)) &= z_0 + \sum_{j \in \Lambda} [(-1)^{j+1}P(\tau_j) + \sum_{i \Rightarrow j} (-1)^{i+1}P(\tau_j + s_i)] \\
&= z_0 + \sum_{j \in \Lambda} [\theta_j P(\tau_j) + \sum_{i \Rightarrow j}(-1)^{i+1}(P(\tau_j+s_i)-P(\tau_j))],
\end{align*}
where $z_0$ is a constant, $s_i=\sigma_i \tau_j^{-2K/d(d+1)}$, and $i \Rightarrow j$ means $i$ is bound to $j$.  In addition, the quantity $\theta_j := (-1)^{j+1} + \sum_{i \Rightarrow j} (-1)^{i+1}$ is never zero because bound indices are always odd and because an index $j \in \Lambda$ has zero or at least two indices bound to it.  See \cite{ChCCC} for details.  

For fixed $\sigma$, we wish to compute the jacobian with respect to $\tau$.  Using the definition of $s_i$ in the previous paragraph, we have
\begin{align*}
\frac{\partial}{\partial \tau_j} [P(\tau_j + s_i(\tau_j,\sigma_i)) - P(\tau_j)] = (P'(\tau_j+s_i)-P'(\tau_j)) - \frac{2K}{d(d+1)} \frac{s_iP'(\tau_j+s_i)}{\tau_j}.
\end{align*}
Now, using multilinearity of the determinant,
\begin{align}
\label{eqn:Jtau+errors}
J(\tau,\sigma) = C_0 J_P(\tau) + \sum \text{error terms},
\end{align}
where $|C_0| = |\pm \prod_{j \in \Lambda} \theta_j|  \sim 1$.  The number of terms in the sum above is $\leq C_d$ and each error term is equal to a constant times a certain determinant, for instance
\begin{align} \label{errorterm_hybrid}
C_{k,l} \det(P'(\tau_{j_1}),\ldots,P'(\tau_{j_{k}}+s_{i(j_k)})-P'(\tau_{j_{k}}),\ldots,\frac{s_{i(j_l)}}{\tau_{j_{l}}} P'(\tau_{j_{l}}+s_{i(j_l)}),\ldots).
\end{align}
Here $\Lambda = \{j_1,\ldots,j_d\}$, each $i(j)$ is bound to $j$, either $k < l$ or $l \leq d$ (so not all entries of the matrix are of the form $P'(\tau_j)$), and
\[
|C_{k,l}| = |(\prod_{i=1}^{k-1} \theta_{j_i})(\frac{-2K}{d(d+1)})^{d-l+1}| \leq C_{d,N}.
\]

The determinants in \eqref{errorterm_hybrid} can be viewed as a hybrid of two types of error terms.  The first of these is
\begin{align} \notag
\det( P'(\tau_{j_1}),\ldots,&P'(\tau_{j_{k-1}}),P'(\tau_{j_k}+s_{i(j_k)})-P'(\tau_{j_k}),\ldots,P'(\tau_{j_d}+s_{i(j_d)})-P'(\tau_{j_d})) \\
\label{errorterm1}
&= \int_{\tau_{j_k}}^{\tau_{j_k}+s_{i(j_k)}} \cdots \int_{\tau_{j_d}}^{\tau_{j_d}+s_{i(j_d)}} \prod_{\ell = k}^d \frac{\partial}{\partial \tau_{j_{\ell}}}|_{\tau_{j_{\ell}} = t_{j_{\ell}}} J_P(\tau) \, dt_{j_d} \cdots dt_{j_k}.
\end{align}
where $k \leq d$ and $i(j)$ is bound to $j$.  The second type is
\begin{align}
\label{errorterm2}
(\prod_{j \in \Lambda'} \frac{s_{i(j)}}{\tau_j}) J_P(t_{i(1)},\ldots,t_{i(d)}).
\end{align}
In \eqref{errorterm2} $t=t(\tau,\sigma)$ and $\emptyset \neq \Lambda' \subset \Lambda$, but now $i(j)$ equals $j$ if $j \notin \Lambda'$ and is bound to $j$ otherwise.  

For clarity of exposition, we will explicitly bound the quantities \eqref{errorterm1} and \eqref{errorterm2}, but our analysis applies equally well to \eqref{errorterm_hybrid}, the hybrid of these terms.

%%%%%%%%%%%%%%%%%%%%%%%%%%%%%%%%%%%%%%%%%%%%%%%%%%
\subsection{Aside:  A geometric identity} \label{sec:geom_id}
Before proceeding in our analysis of the error terms, we record an identity relating the jacobian $J_P$ defined in \eqref{def:J_P} to determinants of certain minors of the matrix $(P'(t),\ldots,P^{(d)}(t))$.  A proof of this identity may be found in \cite{DW}.

For $1 \leq j \leq d$, we define polynomials $L_j=L_{P,j}$ by 
\begin{align} \label{def:L_j}
L_{j}(s) := \det \left( \begin{array}{ccc} P_1'(s) &  \ldots & P_1^{(j)}(s) \\
								\vdots  &  \ddots & \vdots \\
								P_j'(s) &  \ldots & P_j^{(j)}(s)
				\end{array}\right).
\end{align}
Note that $L_d = L_P$, where $L_P$ is the polynomial defined in \eqref{def:L_P}.  

Using this, we recursively define rational functions $J_k:\reals^k \to \reals$ by 
\begin{align} \label{def:J1}
&J_1(t_1) := \frac{L_{d-2}(t_1)L_{d}(t_1)}{L_{d-1}(t_1)^2} \\
\label{def:Jr}
&J_k(t_1,\ldots,t_k) := \\ \notag
& \quad \prod_{j=1}^k \frac{L_{d-k-1}(t_j)L_{d-k+1}(t_j)}{L_{d-k}(t_j)^2} \int_{t_1}^{t_2} \cdots \int_{t_{k-1}}^{t_k} J_{k-1}(s_1,\ldots,s_{k-1})\, ds_1 \cdots ds_{k-1}.
\end{align}
The convention $L_0 \equiv L_{-1} \equiv 1$ is required to define $J_{d-1}$ and $J_d$.  

The algorithm in \cite{DW} begins with an initial decomposition 
\[
\reals = \bigcup_{j=1}^{C_{N,d}} \overline{I_j},
\]
where the $I_j$ are disjoint open intervals, and on each $I_j$, the polynomials $L_1,\ldots,L_{d}$ are all single-signed.  Then for $t = (t_1,\ldots,t_d) \in I_j^d$, we have the identity
\begin{align} \label{eqn:J_P_integral}
J_P(t) = J_d(t).
\end{align}

%%%%%%%%%%%%%%%%%%%%%%%%%%%%%%%%%%%%%%%%%%%%%%%%%%
\subsection{Back to the proof of Lemma \ref{lemma:lb_J}}

We examine a typical instance of \eqref{errorterm1}.  For ease of notation, let the $\tau_j$ be indexed by $j \in \{1,\ldots,d\}$ instead of $\Lambda$.  We expand the identity \eqref{eqn:J_P_integral} 
\begin{align} \label{eqn:J_P_integral2}
J_P(\tau) &= (\prod_{j=1}^d L_{1}(\tau_j) ) \int_{\tau_1}^{\tau_2} \cdots \int_{\tau_{d-1}}^{\tau_d} J_{d-1}(s_1,\ldots,s_{d-1})\, ds_1 \cdots ds_{d-1}.
\end{align}
Let 
\[
h(\tau) := \prod_{j=1}^d L_1(\tau_j) \qquad H(\tau) := J_P(\tau)/h(\tau).  
\]
Then
\begin{align}\label{eq:DJP}
 (\prod_{j \in \Lambda'} &\frac{\partial}{\partial \tau_j}) J_P(\tau) = \sum_{\Lambda'' \subset \Lambda'} T_1 \cdot T_2 :=  \sum_{\Lambda'' \subset \Lambda'} (\prod_{j \in \Lambda' \backslash \Lambda''} \frac{\partial}{\partial \tau_j}) h(\tau) \cdot  (\prod_{j \in \Lambda''} \frac{\partial}{\partial \tau_j})  H(\tau).
\end{align}
We will use the following lemmas to bound $T_1$ and $T_2$ in a typical term from the above sum.

\begin{lemma} \label{lemma:L1'}  With $I\subset (0,\infty)$ and $b=0$ coming from Theorem \ref{thm:DW}, whenever $s \in I$
\begin{align} \label{ineq:L1'}
|{L_1}'(s)| \lesssim \frac{|L_1(s)|}{s}.
\end{align}
\end{lemma}

\begin{lemma} \label{lemma:Id-1} With $I,b$ as in Lemma \ref{lemma:L1'}, whenever $\tau \in I^d$
\begin{align} \label{ineq:DJd}
(\prod_{j \in \Lambda''} \frac{\partial}{\partial \tau_j}) H(\tau) \lesssim \sum_{\delta,\bfi} \frac{|H(\tau)|}{\prod_{j \in \Lambda''} \tau_j^{\delta_j} |\tau_j-\tau_{\bfi(j)}|^{1-\delta_j}},
\end{align}
where the sum is taken over $\delta \in \{0,1\}^{\Lambda''}$ and functions $\bfi:\Lambda'' \to \Lambda$ with $\bfi(j) \neq j$, for all $j \in \Lambda''$.
\end{lemma}

We remark that the conditions on $I$ and $b$, particularly the fact that $I,b$ are determined by the algorithm in \cite{DW}, are crucial hypotheses to these lemmas.  We postpone the proofs of Lemmas \ref{lemma:L1'} and \ref{lemma:Id-1} for a moment and finish proving Lemma \ref{lemma:lb_J}.  

By Lemma \ref{lemma:L1'}
\begin{align} \label{ineq:T1}
|T_1| \lesssim \frac{|h(\tau)|}{\prod_{l \in \Lambda' \backslash \Lambda''} \tau_l}.
\end{align}
Lemma \ref{lemma:Id-1} gives an upper bound for $|T_2|$, and combining that with \eqref{ineq:T1} and \eqref{eq:DJP}, we obtain
\[
|(\prod_{j \in \Lambda'} \frac{\partial}{\partial t_j}) J_P(t_{\Lambda'},\tau_{\Lambda'^c})| \lesssim \sum_{\Lambda'' \subset \Lambda', \delta, \bfi} \frac{|J_P(t_{\Lambda'},\tau_{\Lambda'^c})|}{\prod_{l \in \Lambda' \backslash \Lambda''} t_l \prod_{j \in \Lambda''}(t_j^{\delta_l}|t_j-u_{\bfi(j)}|^{1-\delta_l})},
\]
where $\delta,\bfi$ are as in Lemma \ref{lemma:Id-1} and $u_i=t_i$ if $i \in \Lambda'$ and $\tau_i$ if $i \in \Lambda'^c$.  We return to \eqref{errorterm1}, which is the term we want to estimate.

Let $j \in \Lambda$, let $i$ be bound to $j$, and let $j' \in \Lambda \backslash j$.  Suppose $t(\tau,s) \in \omega$.  By condition (4) of Lemma \ref{lemma:exists_b_struct} and the definition of $s_i$,
\[
|s_i| \leq \eps \delta \alpha_1 [\tau_j\cdot(\tau_j+s_i)]^{-K/d(d+1)}.
\]
Since $t(\tau,s) \in \omega$, both $\tau_j=t_j$ and $\tau_j+s_i=t_i$ are $\gtrsim \alpha_1^n$ (recall $\alpha_2 \geq \alpha_1$).  Raising these lower bounds to the negative power $-K/d(d+1)$ gives an upper bound on $|s_i|$ and using $\tau_j \gtrsim \alpha_1^n$ again implies
\begin{align} \label{ineq:si_tauj}
|s_i| \lesssim \eps \delta \alpha_1^n \lesssim \eps \tau_j.
\end{align}
This implies the inequality (already noted above)
\begin{align} \label{sim:ti_tj}
\tau_j \sim \tau_j+s_i.
\end{align}

We next compare $|\tau_j-\tau_{j'}|$ and $|s_i|$.  If $\tau_{j'} \ll \tau_j$ or if $\tau_j \ll \tau_{j'}$, then 
\[
|\tau_j-\tau_{j'}| \gtrsim \tau_j \gtrsim \eps^{-1} |s_i|,
\]
by \eqref{ineq:si_tauj}.  Say $\tau_j \sim \tau_{j'}$.  Since $j$ and $j'$ must lie in different bands, approximating $\tau_{j'}$ by $\tau_j$ in conclusion (iv) of Lemma \ref{lemma:exists_b_struct},
\[
|\tau_j-\tau_{j'}| \gtrsim \delta \alpha_1 \tau_j^{-2K/d(d+1)} \sim \delta \alpha_1 [\tau_j(\tau_j+s_i)]^{-K/d(d+1)},
\]
where the second estimate follows from \eqref{sim:ti_tj}.  Thus, regardless of the relative sizes of $\tau_j,\tau_{j'}$,
\begin{align} \label{ineq:si_tauj-tauj'}
|\tau_j-\tau_{j'}| \gtrsim  \eps^{-1} |s_i|.
\end{align}
Therefore, we may estimate the integrand in \eqref{errorterm1} by the constant
\[
\sum \frac{|J_P(\tau)|}{\prod_{j \in \Lambda' \backslash \Lambda''} \tau_j \prod_{l \in \Lambda''}(\tau_l^{\delta_l}|\tau_l-\tau_{l'}|^{1-\delta_l})},
\]
which implies that the error term in \eqref{errorterm1} is 
\[
\leq \sum \frac{|J_P(\tau)|\prod_{j \in \Lambda'}|s_j|}{\prod_{j \in \Lambda' \backslash \Lambda''} \tau_j \prod_{l \in \Lambda''}(\tau_l^{\delta_l}|\tau_l-\tau_{l'}|^{1-\delta_l})} \lesssim \eps^{|\Lambda'|} |J_P(\tau)|,
\]
by \eqref{ineq:si_tauj} and \eqref{ineq:si_tauj-tauj'}.

We now have everything we need to estimate \eqref{errorterm2} as well.  By our bound on \eqref{errorterm1}, 
\[
J_P(t_{i(1)},\ldots,t_{i(d)}) \sim J_P(\tau_1,\ldots.\tau_d),
\]
and by \eqref{ineq:si_tauj}, $|\prod_{j \in \Lambda'} s_{i(j)}/\tau_j| \lesssim \eps^{|\Lambda'|}$.

Combining these two estimates, $J(\tau,\sigma) = CJ_P(\tau) + O(\eps)J_P(\tau)$.  It remains to show that $J_P(\tau)$ is bounded from below by the term on the right of \eqref{ineq:lb_J}.  By our assumptions on $I$ (in particular, the conclusions of Theorem \ref{thm:DW}), 
\begin{align*} 
|J_P(\tau)| \gtrsim \prod_{j=1}^d \tau_j^{K/d} \prod_{k<j} |\tau_j-\tau_k|.
\end{align*}
The first term on the right is in the form we want, but we need to work on the second term.  By (iv--v) of Lemma~\ref{lemma:exists_b_struct} and the fact that each quasi-free index is quasi-bound to a unique free index, we have
\begin{align*} 
\prod_{k=1}^d|\tau_k|^{K/d} \prod_{k<l} |\tau_k-\tau_l| &\gtrsim \alpha_1^{d(d-1)/2} (\frac{\beta_1}{\alpha_1})^M (\frac{\alpha_2}{\alpha_1})^{(n+1)(d-1)/2} \prod_{k=1}^d\tau_k^{K/d} \prod_{k<l} (\tau_k \cdot \tau_l)^{-K/d(d+1)} 
\end{align*}
Putting these two inequalities together and using the definition of $n$ to perform a quick computation proves the lower bound claimed in Lemma~\ref{lemma:lb_J}.

%%%%%%%%%%%%%%%%%%%%%%%%%%%%%%%%%%%%%%%%%%%%%%
\section{The proofs of Lemmas~\ref{lemma:L1'} and~\ref{lemma:Id-1}}\label{sec:pf_technical_lemmas}

Our proofs of Lemmas~\ref{lemma:L1'} and~\ref{lemma:Id-1} are not self-contained, but rather consist of some minor adjustments and additions to the arguments of Dendrinos and Wright in Sections~5, 7,~9 of \cite{DW}.  For clarity, in this section we will provide a rough sketch of the decomposition procedure in \cite{DW} before explaining how Lemmas~\ref{lemma:L1'} and~\ref{lemma:Id-1} follow.  

We recall the identity quoted in Section~\ref{sec:geom_id}.  It will be important in what follows to note that one can prove inductively from the definition that $J_{d-1}$ is an antisymmetric function.  

Two decomposition procedures, applied iteratively, constitute the main part of the algorithm in \cite{DW}.  We describe them here.

\fbox{$D1$}:  Let $\eta_1,\ldots,\eta_{d'}$ be complex numbers and $J$ an interval.  This procedure decomposes $J = \bigcup_{i=1}^{d'} \bigcup_{j \neq i} \overline{I_{ij}}$, where the $I_{ij}$ are pairwise disjoint, and each $I_{ij}$ is a union of $O_{d'}(1)$ open intervals.  Each $I_{ij}$ is associated to the real number $b_i = \Re \eta_i$, and for each $i,j$ and $s \in I_{ij}$, 
\begin{align} \label{ineq:characterize_d1_1}
|s-b_i| \leq |s-\eta_k|, \quad 1 \leq k \leq d'\\
\notag |s-\eta_k| \sim A_{ik}|s-b_{i}|^{\delta_{ik}}.
\end{align}
Here $\delta_{ik} \in \{0,1\}$, and $A_{ik} \in \reals$.  

\fbox{$D2$}:  Let $J$ be an interval and $b$ a real number.  Let $\beta_1,\ldots,\beta_{d'}$ be (not necessarily distinct) complex numbers with $|b+\beta_1| \leq \cdots \leq |b+\beta_{d'}|$.  This procedure produces a decomposition 
\[
J = \bigcup_{j=1}^{C_{d'}} \overline{\scriptg_j} \cup \bigcup_{j=1}^{C_{d'}} \overline{\scriptd_j},
\]
where the $\scriptg_j$, $\scriptd_k$ are pairwise disjoint, and each $\scriptg_j$ and each $\scriptd_k$ is a union of $O(1)$ open intervals.  On each $\scriptg_j$ (called a gap),
\begin{align} \label{ineq:characterize_gap}
|s-b - \beta_k| \sim |\beta_k|^{1-\eps_{kj}} |s-b|^{\eps_{kj}} \qquad |s-b-\beta_k| \gtrsim |s-b|,
\end{align}
for some $\eps_{kj} \in \{0,1\}$, and on each $\scriptd_j$ (called a dyadic interval),
\[
|s-b| \sim D_j.
\]

We move now to the implementation of the procedures \fbox{$D1$} and \fbox{$D2$}.  Fix an interval $J$ coming from the initial decomposition above (thus \eqref{eqn:J_P_integral} holds).

{\it{Step 0:}}  Apply \fbox{$D1$} to $J$ with respect to the zeros $\eta_i$ of $L_{d}$.  Fix an interval $I_0$ with corresponding real number $b_0$ from this decomposition.

{\it{Step $n+1$:}}  
Assume that step $n$ has been completed ($0 \leq n \leq d-2$), leaving us with an interval $I_n$ and real number $b_n$.  Apply \fbox{$D2$} to $I_n$ with respect to $b_n$ and the zeros of $L_{n+1}(\cdot + b_n)$.  There are two possibilities for an interval $J$ from this decomposition.
{\it{Case I:}}  $J$ is a gap.  In this case, Step $n+1$ is complete and has produced an interval $I_{n+1}=J$ and real number $b_{n+1}=b_n$.  
{\it{Case II:}}  $J$ is dyadic.  We apply \fbox{$D1$} to $J$ with respect to $b_n$ and the zeros of $L_{P,n+1}$ and $b_0,b_1,\ldots,b_n$ (these are the real numbers corresponding to $I_n$ and its ancestor intervals).  We are left with interval-number pairs.  We choose one of these pairs (arbitrarily) and denote its elements $I_{n+1}$, $b_{n+1}$ to complete step $n+1$.

In either case, the decomposition has been performed so that on $I_{n+1}$,
\begin{align} \label{ineq:LPj_sim}
|L_{d}(s)| \sim A_d |s-b_{n+1}|^{k_{d}}, \: |L_{1}(s)| \sim A_1 |s-b_{n+1}|^{k_1}, \: \ldots , \: |L_{{n+1}}(s)| \sim A_{n+1}|s-b_{n+1}|^{k_{n+1}}.
\end{align}
Here, the constants $A_j$ and the non-negative integer exponents $k_j$ are allowed to change from line to line to reflect the fact that if the interval $J$ is dyadic, then $|s-b_n|$ is nearly constant.  If $|s-b_n|$ is nearly constant, then by the analogue of \eqref{ineq:LPj_sim} from Step $n$, so are $L_{d},L_{1},\ldots,L_{n}$.  

The final step ($n=d$) of the decomposition is to decompose each interval $I_{d-1}$ coming from Step $d-1$ so that none of the subintervals $I_d$ contains any of the real numbers $b_0,b_1,\ldots,b_{d-1}$ associated to $I_{d-1}$ and its ancestors.

A detailed proof of \eqref{ineq:LPj_sim} may be found in \cite{DW}.  To simplify the exposition, we have omitted some crucial details in the sketch above.  For example, linear transformations are used between steps to guarantee certain exponents $k_j$ do not arise which would prevent the deduction of (ii) of Theorem~\ref{thm:DW} from \eqref{ineq:LPj_sim} and \eqref{eqn:J_P_integral}.  We have also made a minor change in the algorithm of Dendrinos and Wright by using the real numbers $b_j$ determined in previous steps to perform the decomposition in Case II of Step $n+1$.  This alteration is miniscule, and moreover, the modified algorithm is needed only to establish the upper bound on the error terms in \eqref{eqn:Jtau+errors}.  Having bounded those error terms, one can use the original algorithm in \cite{DW} to prove Theorem~\ref{thm:DW}, so this change does not cause any technical issues to arise.  In fact, the modified algorithm could also be employed in the proof of Theorem~\ref{thm:DW}.  

\begin{proof}[Proof of Lemma~\ref{lemma:L1'}]
After $d-1$ steps as above, we select an interval $I$ with corresponding real number $b$.  From \eqref{ineq:LPj_sim}, we know that for $s \in I$,
\[
|L_{1}(s)| \sim A_1 |s-b|^{k_1}.
\]
To prove the lemma (wherein things have been arranged so $b=0$, $I\subset(0,\infty)$), we must show that
\begin{align} \label{ineq:L1'_goal}
|L_{1}'(s)| \lesssim \frac{|L_{1}(s)|}{|s-b|}.
\end{align}

We let $b_0,b_1,\ldots,b_{d-1}=b$ be the real numbers corresponding to the ancestor intervals of $I$ in the decomposition procedure.  

We begin our analysis after the application of \fbox{$D2$} in Step 1.  
{\it{Case I:}}  the ancestor $J$ of $I$ arising after \fbox{$D2$} is a gap.  Then $b_1=b_0$ and
\[
L_1(s) = B_1 \prod_i (s-b_0-\beta_i)^{n_i},
\]
where the $\beta_i$ are the complex numbers with respect to which \fbox{$D2$} was performed.  By the product rule and the triangle inequality,
\[
|L_1'(s)| \leq \sum_i \frac{n_i |L_1(s)|}{|s-b_0-\beta_j|}.
\]
Thus, by \eqref{ineq:characterize_gap}, 
\[
|L_1'(s)| \leq C_{N,d} \frac{|L_1(s)|}{|s-b_0|} = C_{N,d} \frac{|L_1(s)|}{|s-b_1|},
\]
where $C_{N,d} = \sum_i n_i$.  
{\it{Case II:}}  the ancestor $J$ of $I$ arising after \fbox{$D2$} is dyadic.  In the notation of \fbox{$D1$},
\[
L_1(s) = B_1 \prod_i (s-\eta_i)^{n_i},
\]
and $b_1$ satisfies \eqref{ineq:characterize_d1_1} with the index $i$ replaced by 1 on $I$.  Hence, by the same arguments as in Case I,
\[
|L_1'(s)| \leq C_{N,d} \frac{|L_1(s)|}{|s-b_1|}.
\]

We wish to prove that after Step $n$, 
\begin{align} \label{ineq:L1'_step_n}
|L_1'(s)| \leq C_{N,d} \frac{|L_1(s)|}{|s-b_n|}.
\end{align}
We proceed by induction, assuming that \eqref{ineq:L1'_step_n} holds after step $n$.  Step $n+1$ begins with a \fbox{$D2$} decomposition.  If after this the ancestor $J$ of $I$ is a gap, we have $b_{n+1}=b_n$.  If the ancestor $J$ of $I$ is dyadic, by the way \fbox{$D1$} is performed in the above version of the algorithm,  $|s-b_{n+1}| \leq |s-b_n|$.  In either case, the analogue of \eqref{ineq:L1'_step_n} holds.

This completes the proof of Lemma~\ref{lemma:L1'}.
\end{proof}

\begin{proof}[Proof of Lemma~\ref{lemma:Id-1}]  
We will prove the lemma by rewriting the left side of \eqref{ineq:DJd} as an integral and then approximating the integral.  

It is easy to see that if $\Lambda'' = \{1,\ldots,d\}$, then $T_2 = 0$.  Let $i_1,\ldots,i_k$ be an increasing enumeration of the elements of $\Lambda \backslash \Lambda''$.  Using antisymmetry of $J_{d-1}$, \begin{align} \label{eq:DJd}
(\prod_{j \in \Lambda''} \frac{\partial}{\partial \tau_j}) H(\tau) = 
\pm \int_{\tau_{i_1}}^{\tau_{i_2}} \cdots \int_{\tau_{i_{k-1}}}^{\tau_{i_k}} J_{d-1}(s_1,\ldots,s_{k-1}, \tau_{\Lambda''})\, ds_1 \cdots ds_{k-1}.
\end{align}
Next, we proceed to the estimation of $J_{d-1}(t)$ for $t \in I^d$.  Recall that we are assuming that $b=0$ and that $I \subset (0,\infty)$.

Since $J_P$ is antisymmetric, by reordering indices if needed, we may assume that $t_1 <  \cdots < t_d$.  Thus the only points $(s_1,\ldots,s_{d-1})$ relevant to the integral in \eqref{eqn:J_P_integral} are those with $s_1 < \cdots < s_{d-1}$.  By induction, for each $1\leq r \leq d$, only those $s \in I^{r-1}$ with $s_1 <\cdots < s_{r-1}$ are relevant in \eqref{def:Jr}.  Because the $s_i$ form a monotone sequence and since the $L_{j}$ are all single-signed on $I$ (by the initial decomposition), each integrand $J_{k-1}$ in \eqref{def:Jr} is single signed on the domain of integration.  This fact will make valid the approximations below.

We will proceed by induction.  In \cite[Section 9]{DW}, the authors define certain integers $\sigma_1,\ldots,\sigma_{d-1}$ such that if $S_{d-1}$ is defined inductively by 
\[
S_r(t_1,\ldots,t_r) := \prod_{s=1}^r t_s^{\sigma_r} \int_{t_1}^{t_2} \cdots \int_{t_{r-1}}^{t_r} S_{r-1}(w_1,\ldots,w_{r-1})\,dw_{r-1}\cdots dw_1,
\]
for $1 <r \leq d-1$, then we have (at the last stage) that $|J_{d-1}| \sim |S_{d-1}|$.  We note in particular that the absolute value of the right hand side of \eqref{eq:DJd} is 
\begin{align} \label{eq:DTd}
\sim \pm \int_{\tau_{i_1}}^{\tau_{i_2}} \cdots \int_{\tau_{i_{k-1}}}^{\tau_{i_k}} S_{d-1}(s_1,\ldots,s_{k-1}, \tau_{\Lambda''})\, ds_1 \cdots ds_{k-1}.
\end{align}

We establish inductively a formula for $S_r$.  To do this, we will use the fact that the algorithm in \cite{DW} ensures that the $\sigma_j$ are such that none of the $S_j$ contains a $t_i^{-1}$ term (and hence none contains a $\log t_i$ term).  Suppose that 
\begin{align} \label{eq:induct}
S_{r-1}(t_1,\ldots,t_{r-1}) = C_{r-1} \sum_{\rho \in \Perm_{r-1}} \rm{sgn}(\rho) t_1^{k_{\rho(1)}} \cdots t_{r-1}^{k_{\rho(r-1)}},
\end{align}
where $C_{r-1} \neq 0$, and the $k_i$ are integers.  (We recall that $\Perm_{r-1}$ is the set of permutations on $r-1$ indices.)  The case $r=2$ of this hypothesis is trivial.  Then
\begin{align*}
T_r(t_1,\ldots,t_r) &:= \int_{t_1}^{t_2} \cdots \int_{t_{r-1}}^{t_r} S_{r-1}(w_1,\ldots,w_{r-1})\,dw_{r-1}\cdots dw_1 \\
&= C' \sum_{\rho \in \Perm_{r-1}} \rm{sgn}(\rho) \prod_{j=1}^{r-1}(t_j^{k_{\rho(j)}+1} - t_{j+1}^{k_{\rho(j)}+1}).
\end{align*}
When we multiply out one of the summands in the last line above, we will have a sum of two types of monomials:  those with no repeated indices, and those with repeated indices.  For example,
\[
(t_1-t_2)(t_2-t_3) = (t_1t_2 - t_1t_3 + t_2t_3) - t_2t_2.
\]
The terms with repeated indices will cancel when we sum on $\rho \in \Perm_{r-1}$, and we obtain
\begin{align*}
T_r(t_1,\ldots,t_r) &= C' \sum_{\rho \in \Perm_{r-1}} \rm{sgn}(\rho) \sum_{j=1}^r (-1)^{r-j} t_1^{k_{\rho(1)}+1} \cdots t_{j-1}^{k_{\rho(j-1)}+1} t_{j+1}^{k_{\rho(j)}+1} \cdots t_r^{k_{\rho(r-1)}+1} \\
&=  C' \sum_{\tau \in \Perm_r} \rm{sgn}(\tau) t_1^{\ell_{\tau(1)}} \cdots t_r^{\ell_{\tau(r)}},
\end{align*}
where $\ell_j = k_j+1$ for $1 \leq j \leq r-1$ and $\ell_r = 0$.  Therefore $S_r(t_1,\ldots,t_r)$ is also of the form \eqref{eq:induct}.

This implies that 
\[
T_{d}(t_1,\ldots,t_{d}) = \det( [t_j^{\alpha_1}, \cdots, t_j^{\alpha_{d}}]_{j=1}^{d}).
\]
This is a Vandermonde-type determinant and is equal to
\[
T_{d}(t_1,\ldots,t_{d}) = C_{P,I} \scriptp(t) \prod_{j=1}^{d} \prod_{i < j} (t_j-t_i),
\]
with $C_{P,I} \sim 1$ and $\scriptp \nequiv 0$ a symmetric polynomial with non-negative coefficients.  For a reference, see \cite[pp. 200-201]{Weyl}; this was used in related contexts in \cite{ChTAMS}, \cite{Gress06}.  

The right side of \eqref{eq:DTd} equals
\[
(\prod_{j \in \Lambda''} \frac{\partial}{\partial \tau_j}) T_d(t),
\]
whose absolute value is easily shown to be 
\[
\lesssim \sum \frac{|T_d(t)|}{\prod_{j \in \Lambda''} t_j^{\delta_j} |t_j-t_{j'}|^{1-\delta_j}} \sim \sum \frac{|H(t)|}{\prod_{j \in \Lambda''} t_j^{\delta_j} |t_j-t_{j'}|^{1-\delta_j}},
\]
where the sum is as in the statement of the lemma.  This completes the proof of Lemma~\ref{lemma:Id-1}.
\end{proof}

%%%%%%%%%%%%%%%%%%%%%%%%%%%%%%%%%%
\section{Proof of Lemma~\ref{lemma:mlF}}

In this section, we modify the techniques in the previous section (using arguments similar to those in \cite{StoCCC}) to prove Lemma~\ref{lemma:mlF}.

Recall that $\gamma_1 := \max\{\alpha_1,\beta_1\}$.  We also define
\[
\gamma_2 := \max\{\alpha_2,\beta_2\}.
\]

The proof of the following lemma is almost exactly the same as the proof of Lemma~\ref{lemma:setup_mlE}.  We leave the details to the reader.

\begin{lemma} \label{lemma:setup_mlF}
There exist $x_0 \in E$, a constant $c > 0$, and measurable sets $\Omega_1 \subset I$, $\Omega_i \subset \Omega_{i-1} \times I$ for $2 \leq i \leq 2d-1$ such that
\begin{itemize}
\item $x_0 + \Phi_i(\Omega_i)$ is contained in $E$ if $i$ is even, $F_2$ if $i=2d-1$ and $F_1$ otherwise.
\item $\mu(\Omega_1)=c \beta_1$, and if $t \in \Omega_{i-1}$, then $\mu\{t_i \in I:(t,t_i) \in \Omega_i\}$ equals:  $c\alpha_1$ if $i$ is even, $c\beta_2$ if $i=2d-1$, and $c\beta_1$ otherwise.
\item $t \in \Omega_i$ implies that $t_i$ is greater than or equal to:  $c \gamma_1^n$ if $i<2d-1$ and $c \beta_2^n$ if $i=2d-1$.  Furthermore, if $j<i<2d-1$, then $|t_i-t_j|$ is greater than or equal to:  $c\beta_1t_j^{-2K/d(d+1)}$ if $i$ is odd and $c\alpha_1 t_j^{-2K/d(d+1)}$ if $i$ is even.  Finally, if $j<i=2d-1$, then $|t_i-t_j|$ is greater than or equal to $c't_i$ if $t_j \leq c' \beta_2^n$ and $c' \beta_2 t_j^{-2K/d(d+1)}$ otherwise.
\end{itemize}
\end{lemma}

As in the proof of Lemma~\ref{lemma:mlE}, the proof of Lemma~\ref{lemma:mlF} breaks into two cases.  

The first case is when $\beta_1 \gtrsim \alpha_1$ (thus $\beta_2 \geq \beta_1 \gtrsim \alpha_1 \geq \alpha_2$ by assumption).  Proceeding as in Lemma~\ref{lemma:initial_lb}, one can show that 
\begin{align} \label{ineq:initial_lb_F2}
|F_2| \gtrsim \alpha_1^{d(d+1)/2} (\frac{\beta_1}{\alpha_1})^{r_{d+1}} (\frac{\beta_2}{\beta_1})^{(d+1)/2+n(d-1)/2}.
\end{align}
Note that the analogue of Lemma~\ref{lemma:elim_diff} for our current setup is the same as the original except that the first two conditions hold when $k<2d-1$, and when $j<k=2d-1$, we have $|t_k-t_j| \gtrsim \beta_2^{(1+n)/2} \beta_1^{(1-n)/2}(t_kt_j)^{-K/d(d+1)}$.

The second case is when $\beta_1 \ll \alpha_1$.  

To begin, we establish a lower bound for $t_{2d-1}$.  {\it A priori}, the best we can do is $t_{2d-1} \gtrsim \beta_2^n$.  We will show that we can assume that $t_{2d-1} \gtrsim \eta^n\alpha_2^n$ ($\eta$ being the quantity in the statement of Lemma~\ref{lemma:mlF}).  

To see this, we let $E_B \subset E$ be the set of $x$ such that
\[
\mu(\{s \geq c \eta^n \alpha_2^n : x+P(s) \in F_2\}) \ll \beta_2.
\]
Here $c \ll n$.  By the hypothesis that $T^*\chi_{F_2} \geq \beta_2$ on $E$, we have that $T^*_B \chi_{F_2} \gtrsim \beta_2$ on $E_B$, where 
\[
T_B f(x) = \int_{[0,c\eta^n \alpha_2^n]} f(x-P(s))\, d\mu(s).
\]
Supposing that $|E_B| \gtrsim |E|$, this implies that
\[
\scriptt_B(E_B,F_2) := \langle \chi_{E_B}, T_B^* \chi_{F_2} \rangle \gtrsim \beta_2 |E| \gtrsim \eta \scriptt(E,F_2).
\]
On the other hand, $T_B \chi_{E}(x) \ll \eta \alpha_2$ for all $x$, so we must have
\[
\scriptt_B(E_B,F_2) \ll \eta \alpha_2 |F_2| = \eta \scriptt(E,F_2).
\]
This is a contradiction, so we must actually have $|E_B| \ll |E|$.  

Let $E_G = E \backslash E_B$.  Then $|E_G| \sim |E|$.  Moreover, on $E_G$,
\[
T\chi_{F_1} \geq \alpha_1 \qquad T_{c\eta^n \alpha_2^n} \chi_{F_2} \gtrsim \alpha_2
\]
(see \eqref{def:T_gamma}).  Thus $\scriptt(E_G,F_1) \gtrsim \beta_1|E| \geq \eta \scriptt(E,F_1)$ and $\scriptt_{c\eta^n \alpha_2^n}(E_G,F_2) \gtrsim \beta_2 |E| \geq \eta \scriptt(E,F_2)$.  Thus with
\[
\tilde{\alpha}_1 := \frac{\scriptt(E_G,F_1)}{|F_1|} \qquad \tilde{\alpha}_2:= \frac{\scriptt_{c\eta^n \alpha_2^n}(E_G,F_2) }{|F_2|},
\]
we have $\alpha_j \geq \tilde{\alpha}_j \gtrsim \eta \alpha_j$.

Henceforth, we will proceed as though $\eta \sim 1$, and hence $\tilde{\alpha}_j \sim \alpha_j$.  The form of the conclusion in \eqref{eq:ineqexp} is such that the general case requires little change in the analysis.

By refining $\Omega_{2d-1}$ if necessary, we may assume that either $t_{2d-1} \gtrsim \alpha_1^n$ or that $t_{2d-1} \ll \alpha_1^n$ throughout $\Omega_{2d-1}$.  We deal with these cases separately.

%%%%%%%%%%%%%%%%%%%%%%%%%%%%%%%%%%%%%%%%%
\subsection{If $t_{2d-1} \ll \alpha_1^n$}  Keeping in mind that we just proved that the lower bound $t_{2d-1} \gtrsim \gamma_2^n$ holds, we perform a band decomposition precisely as in Section~\ref{subsec:combin}, except that we declare the index $2d-1$ to be free from the beginning (note that $t_{2d-1} \ll t_i$ whenever $i < 2d-1$).  Since $|t_j-t_{2d-1}| \sim t_j$ whenever $j < 2d-1$, and since by definition $2d-1$ can have no indices bound to it, the inequalities \eqref{ineq:si_tauj} and \eqref{ineq:si_tauj-tauj'} hold whenever $i$ is bound to $j$ and $j \neq j' \in \Lambda$.  In addition, whenever $j < 2d-1$,
\[
|t_{2d-1}-t_j| \sim t_j \gtrsim \alpha_1^n \gtrsim \alpha_1^{(n+1)/2}\gamma_2^{(1-n)/2}(t_j t_{2d-1})^{-2K/d(d+1)}.
\]
Thus the analogue of Lemma~\ref{lemma:lb_J} implies that 
\[
J(\tau,\sigma) \gtrsim \alpha_1^{d(d-1)/2} (\frac{\beta_1}{\alpha_1})^M (\frac{\gamma_2}{\alpha_1})^{(1-n)(d-1)/2},
\]
which implies that 
\[
|F_2| \gtrsim \alpha_1^{d(d+1)/2}(\frac{\beta_1}{\alpha_1})^{M+\lceil k/2 \rceil} (\frac{\gamma_2}{\alpha_1})^{(1-n)(d-1)/2} ( \frac{\beta_2}{\beta_1}),
\]
by arguments similar to those in Section~\ref{subsec:combin}.  Because $2d-1$, the first index ($2d-k$), and all of the even indices are free, there are at least $\lceil k/2 \rceil +1$ free indices.  As the number of free plus the number of quasi-free indices equals $d$, the exponent of $\frac{\beta_1}{\alpha_1}$ in the above inequality is $\leq d-1$.  From that and the definition of $\gamma_2$,
\begin{align*}
|F_2|&\geq \alpha_1^{d(d+1)/2} (\frac{\beta_1}{\alpha_1})^{d-1} (\frac{\beta_2}{\alpha_1}) (\frac{\alpha_2}{\alpha_1})^{\max\{0,(1-n)(d-1)/2 - 1\}} (\frac{\beta_2}{\beta_1}) \\
&= \alpha_1^{d(d+1)/2} (\frac{\beta_1}{\alpha_1})^{d} (\frac{\beta_2}{\beta_1})^2 (\frac{\alpha_2}{\alpha_1})^{\max\{0,(1-n)(d-1)/2 - 1\}}.
\end{align*}
By checking the cases when the above $\max$ is zero and nonzero separately, one can verify that \eqref{eq:ineqexp} is satisfied, and Lemma~\ref{lemma:mlF} is proved in the case $\beta_1 \ll \alpha_1$ and $t_{2d-1} \ll \alpha_1^n$.

%%%%%%%%%%%%%%%%%%%%%%%%%%%%%%%%%%%%%%%%%
\subsection{If $t_{2d-1} \gtrsim \alpha_1^n$}  We perform a band decomposition.  The following lemma holds:

\begin{lemma} \label{lemma:exists_b_struct2}
Let $\eps > 0$.  Then there exists $c_{d,\eps} < \delta' < \eps \delta < \eps c$, an integer $d \leq k \leq 2d-1$, an element $t_0 \in \Omega_{2d-k-1}$, a set $\omega \subset \{t:(t_0,t) \in \Omega_{2d-1}\}$ with $\mu^k(\omega) \sim \alpha_1^{\lfloor k/2 \rfloor} \beta_1^{\lceil k/2 \rceil} (\beta_2/\beta_1)$, and a band structure on $[2d-k,2d-1]$ such that properties (i--iv) of Lemma~\ref{lemma:exists_b_struct} hold.
\end{lemma}

As with Lemma~\ref{lemma:exists_b_struct}, this can be proved by making a few modifications to the proof of the analogous lemma in \cite{ChCCC} .

Say $2d-1$ is free.  Note that if $\beta_2 \geq \alpha_1$, we have that 
\[
|t_{2d-1}-t_j| \gtrsim \beta_2^{(1+n)/2}\alpha_1^{(1-n)/2}(t_{2d-1}t_j)^{-K/d(d+1)},
\]
as can be shown in a manner similar to the proof of Lemma~\ref{lemma:elim_diff}.  On the other hand, if $\beta_2 < \alpha_1$, then by ({\it{ii}}) of Lemma~\ref{lemma:exists_b_struct}, 
\[
|t_{2d-1}-t_j| \gtrsim \alpha_1(t_{2d-1}t_j)^{-K/d(d+1)} \gtrsim \beta_2(t_{2d-1}t_j)^{-K/d(d+1)}.
\]

In addition, because $t_{2d-1} \gtrsim \alpha_1^n$, the arguments leading up to \eqref{ineq:si_tauj} and \eqref{ineq:si_tauj-tauj'} apply, and we have that whenever $i$ is bound to $j$ and $j \neq j' \in \Lambda$ ($\Lambda$ being the set of free and quasi-free indices)
\[
|t_i-t_j| < \eps t_j \qquad |t_i-t_j| < \eps |t_j-t_{j'}|.
\]
Thus, for sufficiently small $\eps$, the analogue of Lemma~\ref{lemma:lb_J} implies that
\begin{align} \label{ineq:local88}
|J(\tau,\sigma)| \gtrsim \alpha_1^{d(d+1)/2} (\beta_1/\alpha_1)^{M+\lceil k/2 \rceil} (\beta_2/\alpha_1)^{\rho} (\beta_2/\beta_1) \prod_{j=1}^d \tau_j^{2K/d(d+1)},
\end{align}
where $\rho=(1+n)/(d-1)/2$ if $\beta_2 \geq \alpha_1$ and 1 if $\beta_2 < \alpha_1$.  

Since $d \geq 4$, $\rho$ is at least 1 when $\beta_2 \geq \alpha_1$, and \eqref{ineq:local88} holds with $\rho=1$ regardless of the relative magnitudes of $\beta_2$ and $\alpha_1$.  Moreover, $2d-1$, all of the even indices, and the least index ($2d-k$) are free.  Since $M$ plus the number of free indices equals $d$, one can check that the exponent of $\beta_1/\alpha_1$ is at most $d-1$.  Hence
\[
|F_2| \gtrsim \alpha_1^{d(d-1)/2} \beta_1^d (\beta_2/\beta_1)^2.
\]
One can verify that the inequalities and equalities needed for Lemma~\ref{lemma:mlF} are satisfied.  

Finally, suppose that $2d-1$ is not free (hence $\beta_2 < \alpha_1$).  Then one can show that if $t \in \omega$ and $j < 2d-1$, then 
\[
|t_j-t_{2d-1}| \gtrsim \beta_2 (t_jt_{2d-1})^{-K/d(d+1)}.
\]
Indeed, if $j$ and $2d-1$ are in different bands, $|t_j-t_{2d-1}| > \delta \alpha_1 (t_jt_{2d-1})^{-K/d(d+1)}$ and if $j$ and $2d-1$ are in the same band, $t_j \sim t_{2d-1}$ as was shown in Section~\ref{subsec:combin}, and the inequality follows from the construction of $\Omega_{2d-1}$.  

One can produce a two-stage band structure as in \cite{StoCCC} by partitioning the band $\scriptb$ containing $2d-1$.  For completeness, we note that one obtains the following lemma:

\begin{lemma}\label{lemma:2stage_b_struct}  Let $\eps > 0$.  Then there exist parameters $\delta,\delta',\rho,\rho'$ satisfying
\[
0 < c_{d,\eps} < \rho' < \eps \rho \qquad \rho < \delta' \qquad c_{d,\eps} < \delta' < \eps \delta,
\]
a set $\omega \subset \Omega_{2d-1}$ with $\mu^{2d-1}(\omega) \sim \mu^{2d-1}(\Omega_{2d-1})$, and a two stage band structure on $\{1,\ldots,2d-1\}$ satisfying the following:  The first stage is a band structure on $\{1,\ldots,2d-1\}$.  Each even index is free after the first stage.  The second stage is a band structure on the band $\scriptb$ containing $2d-1$.  Let $t \in \omega$.  Consider the bands created in the first stage. 
\begin{itemize}
\item If $i$ and $j$ lie in different bands, then $|t_i-t_j| \geq \delta \alpha_1 (t_it_j)^{-K/d(d+1)}$.
\item If $i$ is quasi-bound to $j$, then $c_n \beta_1(t_it_j)^{-K/d(d+1)} \leq |t_i-t_j| < \delta \alpha_1 (t_it_j)^{-K/d(d+1)}$.
\item If $i$ is bound to $j$, then $|t_i-t_j| < \delta' \alpha_1 (t_it_j)^{-K/d(d+1)}$.
\end{itemize}
Now we let $i,j \in \scriptb$.
\begin{itemize}
\item If $i$ and $j$ lie in different bands, then $|t_i-t_j| \geq \rho \gamma_2 (t_it_j)^{-K/d(d+1)}$.
\item If $i$ is quasi-bound to $j$, then $c_n \beta_1 (t_it_j)^{-K/d(d+1)} \leq |t_i-t_j| < \rho \gamma_2 (t_it_j)^{-K/d(d+1)}$.  If $i=2d-1$ is quasi-bound to $j$, the lower bound is $c_n \beta_2 (t_it_j)^{-K/d(d+1)} < |t_i-t_j|$.
\item If $i$ is bound to $j$, then $|t_i-t_j| \leq \rho' \gamma_2 (t_it_j)^{-K/d(d+1)}$.
\end{itemize}
Here, 
\[
\gamma_2 = \max\{\alpha_2,\beta_2\}.
\]
\end{lemma}

The proof of Lemma~\ref{lemma:2stage_b_struct} is similar to arguments in \cite{StoCCC}, with modifications as in previous sections to handle the measure $\mu$.  Here one uses the fact that $\beta_2 < \alpha_1$ implies that $\gamma_2 < \alpha_1$ (because $\alpha_2 < \alpha_1$), and hence $t_i > \gamma_2^n$ for $i \in \scriptb$.

With Lemma~\ref{lemma:2stage_b_struct} proved, the proof of Lemma~\ref{lemma:mlF} is exactly as in \cite{StoCCC}, with adaptations made to handle the measure $\mu$ as in Section~\ref{subsec:combin}.  One needs an analogue of Lemma~\ref{lemma:lb_J}, but the adaptation is straightforward.  In making this adaptation, it is important to note the following:  letting $\tau,s,\sigma$ be as in Section~\ref{subsec:combin}, then \eqref{ineq:si_tauj} and \eqref{ineq:si_tauj-tauj'} hold whenever $i$ is bound to $j$ and $j' \neq j$ is free or quasi-free.  

Now Lemma~\ref{lemma:mlF} has been proved in all possible cases.

%%%%%%%%%%%%%%%%%%%%%%%%%%%%%%%%%%%%%%%%%%%%%%%%%%%%%%%%%%%%%%%%%%%%%%%%%%%%%%%%%%%%%%%%%%%%


\begin{thebibliography}{99}
%%%%%%%%%%%%%%%%%%%%%%%%%%%%%%%%%%%%%%%%%%%%%%%%%%%%%%%%%%%%%%%%%%%%%%%%%%%%%%%%%%%%%%%%%%%%%%
\bibitem{BOSI}
{J.-G.~Bak, D.~M.~Oberlin, \and A.~Seeger},
{\em Restriction of Fourier transforms to curves and related oscillatory integrals}, 
Amer.\ J.~Math.\ 131 (2009), no.~2, 277--311.
%
\bibitem{BOSII}
{\bysame},
{\em Restriction of Fourier transforms to curves. II. Some classes with vanishing torsion},
J.~Aust.\ Math.\ Soc.\ 85 (2008), no.~1, 1--28.
%
\bibitem{Choi_JAustMS}
{Y.~Choi}, 
{\em The $L^p-L^q$ mapping properties of convolution operators with the affine arclength measure on space curves}, 
J.~Aust.\ Math.\ Soc.\ 75 (2003), no.~2, 247--261.
%
\bibitem{Choi_JKMS}
{\bysame}, 
{\em Convolution operators with the affine arclength measure on plane curves}, 
J.~Korean Math.\ Soc.\ 36 (1999), no.~1, 193--207.
%
\bibitem{ChTAMS}
{M.~Christ},
{\em On the restriction of the Fourier transform to curves: endpoint results and the degenerate case}, Trans.\ Amer.\ Math.\ Soc.\ 287 (1985), no.~1, 223--238.
%
\bibitem{ChCCC}
{\bysame},
{\em Convolution, curvature, and combinatorics. A case study},
Internat.\ Math.\ Res.\ Notices  19 (1998), 1033-1048.
%
\bibitem{ChQex} 
{\bysame}, 
{\em Quasi-extremals for a Radon-like transform}, preprint.
%
\bibitem{DF-GW}
{S.~Dendrinos, M.~Folch-Gabayet, \and J.~Wright}
{\em An affine invariant inequality for rational functions and applications in harmonic analysis}, preprint.
%
\bibitem{DLW} 
{S.~Dendrinos, N.~Laghi, \and J.~Wright},
{\em Universal $L^p$ improving for averages along polynomial curves in low dimensions}, 
J.~Funct.\ Anal.\ 257 (2009), no.~5, 1355--1378.
%
\bibitem{DW}
{S.~Dendrinos \and J.~Wright},
{Fourier restriction to polynomial curves I:  a geometric inequality}, preprint.
%
\bibitem{Drury90}
{S.~W.~Drury}, 
{\em Degenerate curves and harmonic analysis}, 
Math.\ Proc.\ Cambridge Philos.\ Soc.\ 108 (1990), no.~1, 89--96.
%
\bibitem{DMII}
{S.~W.~Drury \and B.~P.~Marshall}, 
{\em Fourier restriction theorems for degenerate curves}, 
Math.\ Proc.\ Cambridge Philos.\ Soc.\ 101 (1987), no.~3, 541--553.
%
\bibitem{DMI}
{\bysame},
{\em Fourier restriction theorems for curves with affine and Euclidean arclengths},
Math.\ Proc.\ Cambridge Philos.\ Soc.\ 97 (1985), no.\ 1, 111--125.
%
\bibitem{GressPolyEndpt}
{P.~T.~Gressman}
{\em $L^p$-improving properties of averages on polynomial curves and related integral estimates},  
Math.\ Res.\ Lett.\ 16 (2009), no.~6, 971--989.
%
\bibitem{Gress06}
{\bysame},
{\em Convolution and fractional integration with
measures on homogeneous curves in $\Bbb R\sp n$},
Math.\ Res.\ Lett.\ 11 (2004), no.\ 5-6, 869--881.
%
\bibitem{Gugg}
{H.~W.~Guggenheimer},
{\em Differential geometry}, 
McGraw-Hill Book Co., Inc., New York-San Francisco-Toronto-London 1963.
%
\bibitem{Laghi}
{N.~Laghi},
{\em A note on restricted X-ray transforms},
Math.\ Proc.\ Cambridge Philos.\ Soc.\ 146 (2009), no.~3, 719--729. 
%
\bibitem{Litt}
{W.~Littman},
{\em $L\sp{p}-L\sp{q}$-estimates for singular integral
operators arising from hyperbolic equations. Partial differential
equations}, (Proc.\ Sympos.\ Pure Math.,\ Vol.\ XXIII, Univ.\
California, Berkeley, Calif.,\ 1971),  Amer.\ Math.\ Soc.,\
Providence,\ R.I.,\ (1973) 479--481.
%
\bibitem{Oberlin_arxiv}
{D.~M.~Oberlin},
{\em Convolution with measures on flat curves in low dimensions}, 
preprint, arXiv:  0911.1471.
%
\bibitem{Oberlin02}
{\bysame}, 
{\em Convolution with measures on polynomial curves}, 
Math.\ Scand., 90 (2002), no.~1, 126--138. 
%
\bibitem{Obe3}
{\bysame}, 
{\em A convolution estimate for a measure on a curve in $\bold R^4$. II}, 
Proc.\ Amer.\ Math.\ Soc.\ 127 (1999), no.~1, 217--221.
%
\bibitem{Obe2}
{\bysame}, 
{\em A convolution estimate for a measure on a curve in $\bold R^4$}, 
Proc.\ Amer.\ Math.\ Soc.\ 125 (1997), no.~5, 1355--1361. 
%
\bibitem{Obe1}
{\bysame},
{\em Convolution estimates for some measures on curves},
Proc.\ Amer.\ Math.\ Soc.\ 99 (1987), no.~1, 56--60.
%
\bibitem{Pan}
{Y.~Pan}, 
{\em $L^p$-improving properties for some measures supported on curves}, 
Math.\ Scand.\ 78 (1996), no.~1, 121--132.
%
\bibitem{Shaf}
{I.~R.~Shafarevich}, 
{\em Basic Algebraic Geometry}, Springer-Verlag, Berlin, 1977.
%
\bibitem{Sjolin}
{P.~Sj\"olin},
{\em Fourier multipliers and estimates of the Fourier transform of measures carried by smooth curves in $R^{2}$},
Studia Math.\ 51 (1974), 169--182.
%
\bibitem{StoCCC}
{B.~Stovall},
{\em Endpoint bounds for a generalized Radon transform}, 
J.~Lond.\ Math.\ Soc.\ (2) 80 (2009), no.~2, 357--374.
%
\bibitem{TW}
{T.~Tao \and J.~Wright},
{\em $L\sp p$ improving bounds for averages along curves},
J.\ Amer.\ Math.\ Soc.\ 16 (2003), no.\ 3, 605--638.
%
\bibitem{Weyl}
{H.~Weyl},
{\em The classical groups},
Princeton Univ.~ Press, Princeton NJ, 1946.
\end{thebibliography}
\end{document}